\documentclass[12pt]{amsart}

\usepackage{graphicx}
\usepackage{fullpage}
\usepackage{etoolbox}
\usepackage[inline]{enumitem}
\usepackage[utf8]{inputenc}

\usepackage{frenchineq}
\usepackage{amsmath}  
\usepackage{amssymb}     
\usepackage{amsthm}  
\usepackage{bbm}    
\usepackage{bm}
\usepackage{accents}
\usepackage{mathrsfs}
\usepackage{mathtools}
\usepackage{fixmath}
\usepackage{fullpage}
\usepackage{tikz}
\usepackage{phdalgo}
\usepackage{microtype}
\usepackage{cleveref}

\usetikzlibrary{shapes}

\newtheorem{theorem}{Theorem}[section]
\newtheorem{proposition}[theorem]{Proposition}

\newtheorem{lemma}[theorem]{Lemma}

\newtheorem{example}{Example}
\theoremstyle{remark}
\newtheorem{remark}{Remark}

\newcommand{\diffsym}{\mathbin{\triangle}}

\usepackage[colorinlistoftodos,bordercolor=orange,backgroundcolor=orange!20,linecolor=orange,textsize=scriptsize]{todonotes}
\newcommand{\R}{\mathbb{R}}
\newcommand{\T}{\mathbb{T}}

\newcommand{\ie}{\text{i.e.}}
\newcommand{\eg}{\text{e.g.}}
\newcommand{\Pcal}{\mathcal{P}}
\newcommand{\Qcal}{\mathcal{Q}}

\newcommand{\tplus}{\oplus}
\newcommand{\tsum}{\bigoplus}
\newcommand{\tdot}{\odot}
\newcommand{\trans}[1]{{#1}^\top}
\newcommand{\trop}{\mathrm{trop}}
\newcommand{\tro}[1]{\widehat{#1}}
\DeclareMathOperator{\diag}{\mathrm{diag}}
\newcommand{\cla}[1]{#1}
\newcommand{\norm}[1]{\overline{\cla #1}}

\newcommand{\nashA}{P}
\newcommand{\nashB}{Q}

\makeatletter
\@namedef{subjclassname@2020}{%
  \textup{2020} Mathematics Subject Classification}
\makeatother

\title{Tropical complementarity problems\\
and Nash equilibria}

\author{Xavier Allamigeon \and Stéphane Gaubert \and Frédéric Meunier}

\address[Xavier Allamigeon]{INRIA \& CMAP, CNRS, {\'E}cole Polytechnique, Institut Polytechnique de Paris \and CERMICS, \'Ecole des Ponts ParisTech, France.}
\email{xavier.allamigeon@inria.fr}

\address[Stéphane Gaubert]{INRIA \& CMAP, CNRS, {\'E}cole Polytechnique, Institut Polytechnique de Paris, France.}
\email{stephane.gaubert@inria.fr}

\address[Frédéric Meunier]{CERMICS, \'Ecole des Ponts ParisTech, France.}
\email{frederic.meunier@enpc.fr}

\thanks{The first and last authors were partially supported by ANR CAPPS (ANR-17-CE40-0018).}

\begin{document}

\begin{abstract}
Linear complementarity programming is a generalization of linear programming which encompasses the computation of Nash equilibria for bimatrix games. While the latter problem is PPAD-complete, 
we show that the tropical analogue of the complementarity problem associated with Nash equilibria can be solved in polynomial time.  
Moreover, we prove that the Lemke--Howson algorithm carries over the tropical setting and performs a linear number of pivots in the worst case. A consequence of this result is a new class of (classical) bimatrix games for which Nash equilibria computation can be done in polynomial time.
\end{abstract}



\maketitle

\section{Introduction}

Given a square matrix $\cla M$ and a vector $\cla q$, the {\em linear complementarity problem} (LCP) consists in finding a solution $(w,z)$ of the following system
\begin{equation}\label{LCP}\tag{LCP}
\left\{\begin{array}{l}
w = \cla M z + \cla q \\
\trans{w} z = 0 \\
w, z \geq 0 \, .
\end{array}\right.
\end{equation}
It is a central problem in mathematical programming, which is related to the computation of equilibria in economics and convex quadratic minimization, and which encompasses linear programming as a special case. In the full generality given here, deciding whether the system admits a solution is an NP-complete problem \cite{chung1989}. 

Consider the special case where the following two conditions are satisfied simultaneously:
\begin{enumerate}[label=(\roman*)]
\item\label{i} each column of $\cla M$ is nonpositive and has at least one negative entry;
\item\label{ii} each entry of $\cla q$ is positive.
\end{enumerate}
Condition~\ref{ii} implies that the system~\eqref{LCP} admits $(\cla q,0)$ as trivial solution. A well-known theorem ensures that there always exists another solution; see~\cite{ComplementarityBook}. Finding such a solution contains the computation of Nash equilibria in bimatrix games (the principle of the reduction is recalled in Section~\ref{sec:tropical_nash}). We call thus the {\em Nash equilibrium complementarity problem} the problem of finding a solution to
\begin{equation}\label{NECP}\tag{NECP}
\left\{\begin{array}{l}
w = \cla M z + \cla q \\
\trans{w} z=0 \\
z \neq 0 \\
w, z \geq 0 \, ,
\end{array}\right.
\end{equation}
when $\cla M$ and $\cla q$ satisfy the conditions~\ref{i} and~\ref{ii} above. Finding such a solution different from the trivial one is a PPAD-complete problem \cite{chen2009settling}. 

Tropical analogues of mathematical programming problems and related topics have recently been the subject of several works. These cover linear programming and generalizations~\cite{AllamigeonSIDMA15,GaubertJSC,loho2020abstract,loho2019tropical}, semidefinite programming~\cite{allamigeon2018solving,yu_semidefinite_cone}, integer linear programming~\cite{ButkovicMacCaig2014}, and convex constraint satisfaction~\cite{BodirskyMamino2018}. The objective of this work is to introduce the tropical analogues of the problems~\eqref{LCP} and~\eqref{NECP} and to check whether their complexity status is identical in the tropical world. As far as we know, this is the first work addressing tropical complementarity problems. Here, we use the \emph{max-plus tropical semifield} $\T \coloneqq \R \cup \{-\infty\}$, equipped with the addition $x \tplus y \coloneqq \max(x,y)$ and the multiplication $x \tdot y \coloneqq x + y$. (We use the convention that $x + (-\infty) = (-\infty) + x  = -\infty$ for all $x \in \T$.) The zero and unit elements are respectively $-\infty$ and $0$. The tropical operations extend to matrices (and vectors): the tropical sum of two matrices is the tropical componentwise sum, while $(M \tdot N)_{ij} \coloneqq \tsum_k M_{ik} \tdot N_{kj}$. We shall also denote by $-\infty$ the all-$-\infty$ vector and the all-$-\infty$ matrix. The following problem, which we call the {\em tropical linear complementarity problem}, is the natural analogue of the linear complementarity problem in the tropical semifield: given two square matrices $M^-,M^+$ and two vectors $q^-,q^+$ with entries in $\T$ and such that $\min(M_{ij}^+,M_{ij}^-)=\min(q^+_i,q^-_i)=-\infty$ for all $i,j$, find $(w,z)$ with entries in $\T$ such that
\begin{equation}\label{TLCP}\tag{TLCP}
\left\{\begin{array}{l}
w \tplus M^- \tdot z \tplus q^-=M^+ \tdot z \tplus q^+ \\
\trans{w} \tdot z=-\infty \,.
\end{array}\right.
\end{equation}
Let us comment on why~\eqref{TLCP} is the tropical analogue of~\eqref{LCP}. The quantities $M_{ij}^+$ and $M_{ij}^-$ (resp.~$q_i^+$ and $q_i^-$) can be interpreted as the tropical positive and negative parts of a same quantity; see~\cite[Section~2.1.1]{AllamigeonSIDMA15} for further explanations. This justifies the assumption $\min(M_{ij}^+,M_{ij}^-) =\min(q^+_i,q^-_i)=-\infty$, and explains why the constraint $w \tplus M^- \tdot z \tplus q^- = M^+ \tdot z \tplus q^+$ is the tropical counterpart of the constraint $w = Mz + q$ in~\eqref{LCP}. The tropical analogues of the nonnegativity conditions over $w$ and $z$ in~\eqref{LCP} write $w, z \geq -\infty$ in the tropical setting; hence they are implicitly satisfied and they can be omitted. 
We show that, in the general case, 
the complexity status of the linear complementarity problem remains the same when we carry it into the tropical setting.
\begin{proposition}\label{prop:tlcp}
Deciding whether the tropical linear complementarity problem~\eqref{TLCP} has a solution is \textup{NP}-complete, even if all entries of $M^-$ and $q^+$ are equal to $-\infty$.
\end{proposition}

The case where all entries of $M^+$ are equal to $-\infty$ is trivial. Indeed, if $q^-$ is not the $-\infty$ vector, then there is no solution. Otherwise, the system always has a solution provided by $w = q^+$ and $z$ as the $-\infty$ vector. Finding another solution precisely corresponds to the tropical analogue of the Nash equilibrium complementarity problem, which we call the {\em tropical Nash equilibrium complementarity problem}: 
\begin{equation}\label{TNECP}\tag{TNECP}
\left\{\begin{array}{l}
w \tplus M^- \tdot z= q^+ \\
\trans{w} \tdot z=-\infty \\
z \neq -\infty \, ,
\end{array}\right.
\end{equation}
with the following assumptions:
\begin{enumerate}[label=(\roman*$^{\text{trop}}$)]
\item\label{it} no column of $M^-$ is the $-\infty$ vector; 
\item\label{iit} $q^+$ has no $-\infty$ entry.
\end{enumerate}

Perhaps surprisingly, the complexity status of this problem differs from its classical analogue.
\begin{theorem}\label{thm:tnecp}
The tropical Nash equilibrium complementarity problem~\eqref{TNECP} always admits a solution and such a solution can be computed in polynomial time.
\end{theorem}
The proof relies on finding in a bipartite graph a perfect matching distinct from a given one.

The standard approach to solve the classical Nash equilibrium complementarity problem~\eqref{NECP} is the Lemke--Howson algorithm~\cite{lemke1964equilibrium}. We show that this algorithm with the suitable notion of tropical bases solves~\eqref{TNECP} as well. Like in the classical case, this requires that the instance be nondegenerate. (The nondegeneracy of an instance is defined in Section~\ref{sec:tnecp}.) A major difference is that the number of iterations is linear in the tropical case, while it can be exponential in the classical case~\cite{SavaniStengel06}:
\begin{theorem}\label{th:complexity}
The Lemke--Howson algorithm solves nondegenerate instances of~\eqref{TNECP} within at most $2n-1$ iterations, where $n$ is the number of rows of the system. 	
\end{theorem}
As noted in Remark~\ref{remark:nondegenerate} (end of Section~\ref{subsec:trop_LH}), the Lemke--Howson algorithm can actually handle any instance of~\eqref{TNECP} by applying a simple symbolic perturbation. 

One interest of the study of~\eqref{TNECP} is to bring new classes of the Nash equilibrium complementarity problem that can be solved within the same complexity. More precisely, it is usual to map a classical problem to a tropical one by taking the \emph{logarithmic image}: the numerical inputs of the latter are defined as the (signed) logarithm of the numerical inputs of the former. In particular, the logarithmic image of an instance of~\eqref{NECP} is the instance of~\eqref{TNECP} in which $M^- = \log (- \cla M)$ and $q^+ = \log \cla q$. We identify a ``dominance condition'' under which the solutions of both problems are in one-to-one correspondence through their supports. The support of a vector is the index set of its nonzero entries in the classical setting, and of its non-$(-\infty)$ entries in the tropical setting.
The dominance condition, stated in Section~\ref{subsec:realization}, is a diagonal dominance property which holds for a collection of submatrices. It is reminiscent of the notion of lopsidedness in tropical geometry~\cite{purbhoo2008}. It can be decided in polynomial time.
\begin{theorem}\label{th:support}
Consider an instance of~\eqref{NECP} that satisfies the ``dominance condition'' and the corresponding instance of~\eqref{TNECP} obtained by the logarithmic image.
Then the solutions of the two problems have the same supports. 
\end{theorem}
The combination of Theorems~\ref{thm:tnecp} and~\ref{th:support} entails that, under the dominance condition, the Nash equilibrium complementarity problem can be solved in polynomial time. In fact, we refine this result by showing that, on such instances, the Lemke--Howson algorithm follows the same path when applied to the instance of~\eqref{NECP} and its logarithmic image; see Theorem~\ref{th:classical_LH}.

Applying the dominance condition to bimatrix games provides new families of instances for which Nash equilibria can be computed in polynomial time. An example is given by the following proposition. (We refer to Section~\ref{sec:tropical_nash} for a definition of a Nash equilibrium of a bimatrix game.)
\begin{proposition}\label{prop:spec-poly}
Consider a bimatrix game where the payoff matrices of the row and column players are $P \in \R_{\geq 0}^{r \times s}$ and $Q \in \R_{\geq 0}^{r \times s}$ respectively.

If every column of $\nashA$ (resp.~$\trans{\nashB}$) has an entry that is $r-1$ (resp.~$s-1$) times larger than any other entry in the column, the computation of a (classical) Nash equilibrium can be done in polynomial time.
\end{proposition}

We show in Section~\ref{sec:tropical_nash} that the solutions of~\eqref{TNECP} are in correspondence with tropical analogues of Nash equilibria in the case where the payoffs are nonnegative. We discuss the comparison with alternative definitions of Nash equilibria and the generalization to signed payoffs.

The paper is organized as follows. We prove Proposition~\ref{prop:tlcp} in Section~\ref{sec:tlcp}, and Theorem~\ref{thm:tnecp} in Section~\ref{sec:tnecp}. In Section~\ref{sec:trop_bases}, we introduce tropical bases, show how they relate with the feasible bases of classical systems of linear inequalities under the dominance condition, and prove Theorem~\ref{th:support}. In Section~\ref{sec:trop-LH}, we explain how the Lemke--Howson algorithm applies to instances of~\eqref{TNECP}, and show the related complexity results including Theorem~\ref{th:complexity}. Finally, in Section~\ref{sec:tropical_nash}, we discuss the tropicalization of Nash equilibria, and we prove \Cref{prop:spec-poly}.

\section{NP-completeness of tropical linear complementarity}\label{sec:tlcp}

Consider a propositional formula in the conjunctive normal form $\Phi=\bigwedge_{\gamma=1}^p C_{\gamma}$, with literals $x_1,\ldots,x_n$, and the following system over $\T$, whose variables are $w_0,\ldots,w_{2n+p+1},$ and $z_0,\ldots,z_{2n+p+1}$:
\begin{align}
w_0 \tplus 0 & = \displaystyle{\bigoplus_{i=1}^{2n} z_i }\label{sys1} \\
w_i & = z_{n+i} && \forall i\in[n] \label{sys2} \\
w_{n+i} & =  z_i && \forall i\in[n]\\
\displaystyle{w_{2n+\gamma} \tplus 0} & = \displaystyle{\bigoplus_{x_i \in C_{\gamma}} z_i \tplus \bigoplus_{\neg x_i \in C_{\gamma}} z_{n+i}} && \forall\gamma\in[p] \smallskip \label{sys4} \\
w_{2n+p+1} \tplus 0 & = z_0  \label{sys_z0} \\
w_i\tdot z_i &  = -\infty && \forall i\in\{0,\ldots,2n+p+1\} \label{sys5} \, ,
\end{align}
where we denote $[k] \coloneqq \{1, \dots, k\}$. This system is a special case of the tropical linear complementarity problem~\eqref{TLCP}, with all entries of $M^-$ and $q^+$ being equal to $-\infty$. In particular, variables $w_i$ (resp.~$z_i$) have to appear in the left-hand side (resp.~right-hand side) of the constraints. This explains the over-complicated form of the system.

\begin{lemma}\label{lem:eq}
The tropical system above admits a solution if and only if $\Phi$ is satisfiable.
\end{lemma}

\begin{proof}
Suppose first that the tropical system admits a solution $w_0,\ldots,w_{2n+p+1},z_0,\ldots,z_{2n+p+1}$. We start by making an observation that will be useful for showing that $\Phi$ is satisfiable: $w_0 = -\infty$. This can be seen by combining~\eqref{sys5} and the fact that $z_0 \geq 0$ by~\eqref{sys_z0}. Now, set 
$x_i$ to \texttt{true} if and only if $z_i$ is equal to $0$. We check that we obtain a feasible assignment of the literals $x_i$. Consider a clause $C_{\gamma}$. If $z_i=0$ for at least one $x_i$ in $C_{\gamma}$, then $C_{\gamma}$ is satisfied. Thus, we assume that there is no such $z_i$.  We have then $\bigoplus_{x_i \in C_{\gamma}} z_i < 0$ since every $z_i$ ($i \in [n]$) is at most $0$ by~\eqref{sys1} and the equality $w_0 = -\infty$. By~\eqref{sys4}, there is an $\neg x_i$ in $C_{\gamma}$ with $z_{n+i}\geq 0$. Thanks to~\eqref{sys1}, we have actually $z_{n+i}=0$ (again, we use the equality $w_0 = -\infty$). The combination of~\eqref{sys2} and~\eqref{sys5} implies then that $z_i=-\infty$, which corresponds to $x_i$ set to \texttt{false}. In any case, the clause $C_{\gamma}$ is satisfied.

Conversely, suppose that there is a feasible assignment of the literals $x_i$. For every $i\in [n]$, 
\begin{itemize}
\item we set $z_i$ and $w_{n+i}$ to $0$ when $x_i$ is \texttt{true}, and to $-\infty$ otherwise;
\item we set $z_{n+i}$ and $w_i$ to $0$ when $x_i$ is \texttt{false}, and to $-\infty$ otherwise.
\end{itemize}
We set $w_{2n+\gamma}$ and $z_{2n+\gamma}$ to $-\infty$ for every $\gamma \in [p]$, as well as $w_{2n+p+1}$ and $z_{2n+p+1}$. Finally, we set $w_0$ to $-\infty$ and $z_0$ to $0$. Checking that we get a solution of the tropical system above is immediate.
\end{proof}

\begin{proof}[Proof of Proposition~\ref{prop:tlcp}]
The problem of deciding whether a propositional formula in the conjunctive normal form is satisfiable is NP-complete. The tropical system above being a special case of the tropical linear complementarity problem~\eqref{TLCP}, Lemma~\ref{lem:eq} implies that this latter problem is NP-complete as well.
\end{proof}

In the classical setting, there is a polynomial-time reduction from  linear complementarity problems to colorful linear programming~\cite[Section~3.2]{meunier2018colorful}. We point out that this reduction carries over to the tropical setting. Consider a linear system $A^+ \tdot x \tplus b^+ = A^- \tdot x \tplus b^-$ where $A^\pm \in \T^{n \times d}$, $b^\pm \in \T^n$, and $\min(A^+_{ij}, A^-_{ij}) = \min(b^+_i, b^-_i) = -\infty$ for all $i, j$. Let $C_1, \dots, C_k$  be a partition of $[d]$. The \emph{signed tropical colorful linear programming problem}, introduced by Loho and Sanyal~\cite{loho2019tropical}, consists in determining if there exists a solution $x$ of the system whose support contains at most one element of each $C_i$. The problem~\eqref{TLCP} can be transformed into such a problem in which each class $C_i$ consists of the indices of the variables $w_i$ and $z_i$. Therefore, Proposition~\ref{prop:tlcp} allows to recover~\cite[Corollary~4.6]{loho2019tropical} which states that signed tropical colorful linear programming is NP-complete.

\section{Solving the tropical Nash equilibrium complementarity problem in polynomial time}
\label{sec:tnecp}\label{subsec:tncep}

We consider an instance of~\eqref{TNECP}, and, for the sake of brevity, in this section we denote $M^-$ and $q^+$ by $M$ and $q$ respectively. We also denote by~$n$ the number of rows of $M$ and $q$.

We introduce the following colored bipartite (multi)graph $G = (V,E)$, which consists of \emph{row nodes} $u_i$ and \emph{column nodes} $v_j$ ($i, j \in [n]$), and the following set of edges:
\begin{itemize}
\item a blue edge $u_i v_i$ for all $i$;
\item a red edge $u_i v_j$ for all $i, j$ such that $q_i - M_{ij}$ is minimal among the $q_k - M_{kj}$, $k \in [n]$. 
\end{itemize}
Observe that, given $j \in [n]$, even if some $q_i - M_{ij}$ is equal to $+\infty$, the minimum of the terms $q_k - M_{kj}$ ($k \in [n]$) is well-defined and finite, because the conditions~\ref{it} and~\ref{iit} are satisfied. We also remark that every column node $v_j$ in $G$ is incident to at least one red edge, and its degree is at least $2$. 

Given a subset $F \subset E$ of edges $G$, we define the point $\alpha(F) \coloneqq (w,z) \in \T^n \times \T^n$, where
\begin{itemize}
\item $w_i = q_i$ if there is a blue edge between $u_i$ and $v_i$ in $F$, and $w_i = -\infty$ otherwise;
\item $z_j = q_i - M_{ij}$ if there is a red edge between $u_i$ and $v_j$ in $F$, and $z_j = -\infty$ otherwise.
\end{itemize}

\begin{lemma}\label{lemma:feasibility}
If $F \subset E$ covers all row nodes of $G$, then the point $(w,z) = \alpha(F)$ satisfies $w \tplus M \tdot z = q$.  
\end{lemma}

\begin{proof}
Let $(w,z) \coloneqq \alpha(F)$, and consider $i \in [n]$. We trivially have $w_i \leq q_i$. Moreover, for every $j \in [n]$ such that $z_j > -\infty$, there exists $i' \in [n]$ such that  $z_j = q_{i'} - M_{i'j} \leq q_i - M_{ij}$ by construction of $G$. Therefore, $w_i \tplus M_i \tdot z \leq q_i$, and equality holds because $F$ contains an edge of the form $u_i v_j$. If this edge is blue, then $w_i = q_i$, and if it is red, $M_{ij} \tdot z_j = q_i$. We deduce that $w \tplus M \tdot z = q$.
\end{proof}

The following lemma shows that finding a solution of~\eqref{TNECP} reduces to finding a nontrivial perfect matching in $G$. The two solution sets are in one-to-one correspondence in the case where the instance of~\eqref{TNECP} is nondegenerate. We say that an instance of~\eqref{TNECP} is \emph{nondegenerate} if, for each $j \in [n]$, the minimum of the terms $q_k - M_{kj}$ ($k \in [n]$) is attained exactly once. Equivalently, every column node in $G$ is incident to precisely one red edge. 

\begin{lemma}\label{lemma:sol_TNECP}
If $F \subset E$ is a perfect matching containing at least one red edge, then $\alpha(F)$ is a solution of~\eqref{TNECP}. Moreover, if the instance is nondegenerate, any solution of~\eqref{TNECP} arises in this way.
\end{lemma}

\begin{proof}
Let $(w,z) \coloneqq \alpha(F)$. We know that $w \tplus M \tdot z = q$ by Lemma~\ref{lemma:feasibility}. Moreover, for all $j$, the column node $v_j$ has degree $1$ in $F$. Thus $v_j$ has no red or no blue incident edge, which means that $w_j = -\infty$ or $z_j = -\infty$. This implies $w_j \tdot z_j = -\infty$. Finally, $z \neq -\infty$ since there is at least one red edge in $F$. We deduce that $\alpha(F)$ is a solution of~\eqref{TNECP}.

We now suppose that the instance is nondegenerate. Let $(w,z)$ be a solution of~\eqref{TNECP}. We define $F$ as the subset of $E$ consisting of the edges $u_i v_i$ if $w_i > -\infty$, and $u_i v_j$ if $z_j = q_i - M_{ij}$. Every row node is incident to at least one edge of $F$ because $(w,z)$ satisfies $w \tplus M \tdot z = q$. In addition, every column node is incident to at most one red edge, since the instance is nondegenerate. As $w_j \tdot z_j = -\infty$ for all $j$, no column node can be simultaneously incident to a blue edge and a red edge of $F$. We deduce that $F$ is a perfect matching.
\end{proof}

\begin{example}\label{ex:tnecp}
Consider the following instance of~\eqref{TNECP} over $(w,z) \in \T^4 \times \T^4$:
\[
\left\{
\begin{aligned}
& \begin{array}{ccccllll@{{}={}}c}
w_1 & & & & {}\tplus z_1 & {}\tplus ((-2) \tdot z_2) & {}\tplus (3 \tdot z_3) & {}\tplus ((-5) \tdot z_4) & 0 \\
& w_2 & & & {}\tplus (4 \tdot z_1) & & {}\tplus ((-2) \tdot z_3) & & 0 \\
& & w_3 & & {}\tplus (2 \tdot z_1) & {}\tplus z_2 & & {}\tplus ((-1) \tdot z_4) & 0 \\ 
& & & w_4 & & {}\tplus z_2 &  & {}\tplus ((-1) \tdot z_4) & 0
\end{array} \\[.5ex]
& \trans{w} \tdot z = -\infty \\
& z \neq -\infty
\end{aligned}\right.
\]
The associated graph is provided in \Cref{fig:proof}(a). The perfect matching $F$ depicted in \Cref{fig:proof}(d) provides the solution $\alpha(F) = (w,z)$ where $w = \begin{psmallmatrix} -\infty \\ -\infty \\ -\infty \\ 0 \end{psmallmatrix}$ and $z = \begin{psmallmatrix} -4 \\ 0 \\ -3 \\ -\infty \end{psmallmatrix}$.
\end{example}

\begin{figure}
\begin{center}
\begin{tikzpicture}[scale=1.1,
row/.style={regular polygon,regular polygon sides=4, inner sep=-1cm, minimum size=1cm,thick,draw},
col/.style={circle, inner sep=-.2cm, minimum size=.75cm,thick,draw},
be/.style={blue!50,very thick,draw},
re/.style={red!80,very thick,draw}
]
\begin{scope}[xscale=1.5,yscale=.9]
\node[row] (u1) at (0,0) {$u_1$};
\node[row] (u2) at (1,0) {$u_2$};
\node[row] (u3) at (2,0) {$u_3$};
\node[row] (u4) at (3,0) {$u_4$};

\node[col] (v1) at (0,-1.8) {$v_1$};
\node[col] (v2) at (1,-1.8) {$v_2$};
\node[col] (v3) at (2,-1.8) {$v_3$};
\node[col] (v4) at (3,-1.8) {$v_4$};

\draw[be] (u1) -- (v1) (u2) -- (v2) (u3) -- (v3);
\draw[be] (u4) -- (v4);
\draw[re] (u1) -- (v3) (u2) -- (v1);
\draw[re] (v2) -- (u3);
\draw[re] (v2) -- (u4);
\draw[re] (v4) -- (u3);
\draw[re] (v4) to[bend right] (u4);

\node[anchor=base] at (1.5,-3) {(a)};
\end{scope}
\begin{scope}[shift={(8,0)},xscale=1.5,yscale=.9]
\node[row] (u1) at (0,0) {$u_1$};
\node[row] (u2) at (1,0) {$u_2$};
\node[row] (u3) at (2,0) {$u_3$};
\node[row] (u4) at (3,0) {$u_4$};

\node[col] (v1) at (0,-1.8) {$v_1$};
\node[col] (v2) at (1,-1.8) {$v_2$};
\node[col] (v3) at (2,-1.8) {$v_3$};
\node[col] (v4) at (3,-1.8) {$v_4$};

\draw[be] (u1) -- (v1) (u2) -- (v2) (u3) -- (v3);
\draw[be] (u4) -- (v4);
\draw[re] (u1) -- (v3) (u2) -- (v1);
\draw[re] (v2) -- (u3);
\draw[re] (v4) -- (u3);

\node[anchor=base] at (1.5,-3) {(b)};
\end{scope}
\begin{scope}[shift={(0,-4.2)},xscale=1.5,yscale=.9]
\node[row] (u1) at (0,0) {$u_1$};
\node[row] (u2) at (1,0) {$u_2$};
\node[row] (u3) at (2,0) {$u_3$};
\node[row] (u4) at (3,0) {$u_4$};

\node[col] (v1) at (0,-1.8) {$v_1$};
\node[col] (v2) at (1,-1.8) {$v_2$};
\node[col] (v3) at (2,-1.8) {$v_3$};
\node[col] (v4) at (3,-1.8) {$v_4$};

\draw[be] (u1) -- (v1) (u2) -- (v2) (u3) -- (v3);
\draw[re] (u1) -- (v3) (u2) -- (v1);
\draw[re] (v2) -- (u3);

\node[anchor=base] at (1.5,-3) {(c)};
\end{scope}
\begin{scope}[shift={(8,-4.2)},xscale=1.5,yscale=.9]
\node[row] (u1) at (0,0) {$u_1$};
\node[row] (u2) at (1,0) {$u_2$};
\node[row] (u3) at (2,0) {$u_3$};
\node[row] (u4) at (3,0) {$u_4$};

\node[col] (v1) at (0,-1.8) {$v_1$};
\node[col] (v2) at (1,-1.8) {$v_2$};
\node[col] (v3) at (2,-1.8) {$v_3$};
\node[col] (v4) at (3,-1.8) {$v_4$};

\draw[be] (u4) -- (v4);
\draw[re] (u1) -- (v3) (u2) -- (v1);
\draw[re] (v2) -- (u3);

\node[anchor=base] at (1.5,-3) {(d)};
\end{scope}
\end{tikzpicture}
\end{center}
\caption{Illustration of the proof of \Cref{thm:tnecp} on the instance of \Cref{ex:tnecp}. (a) The graph associated with the instance. (b) A subgraph $G'$ in which every column node is incident to precisely one red edge. (c) A cycle $C$ in $G'$. (d) The symmetric difference of the set $F_0$ of blue edges with the cycle~$C$.}\label{fig:proof}
\end{figure}
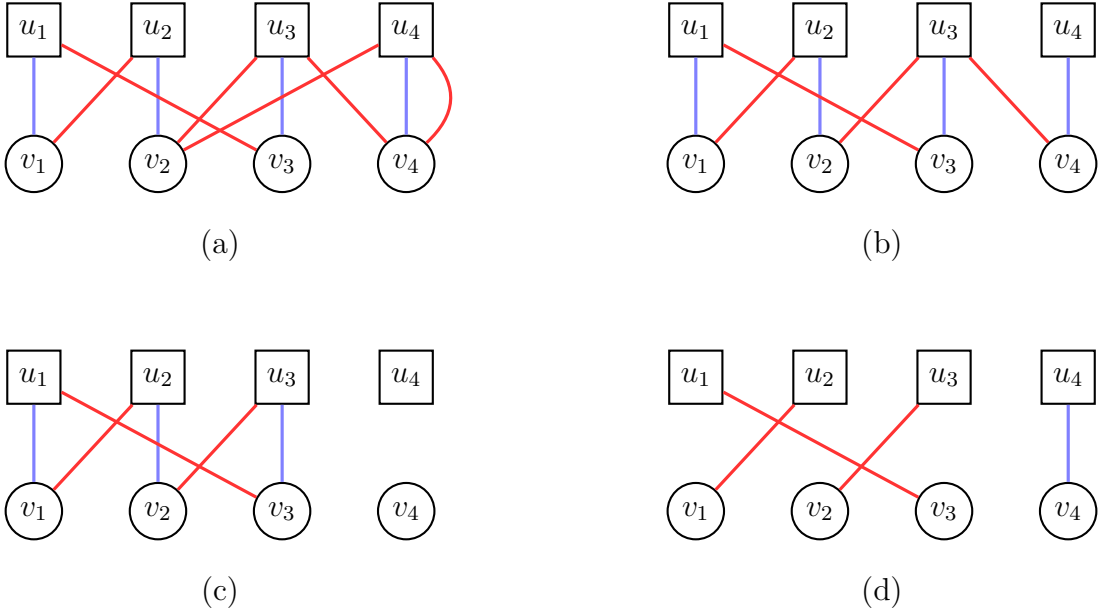

The proof of Theorem~\ref{thm:tnecp} relies on the fact that the graph $G$ contains a cycle that alternates between blue and red edges. We refer to \Cref{fig:proof} for an illustration of the steps of the proof. In the paper, the term ``cycle'' refers to elementary cycles. 
\begin{proof}[Proof of Theorem~\ref{thm:tnecp}]
The set $F_0$ of blue edges forms a perfect matching in $G$. Up to removing red edges in $G$, we can assume that every column node is incident to precisely one red edge. In this case, the graph $G$ has $2n$ edges and $2n$ vertices. Thus, there is a cycle $C$. Moreover, the edges in $C$ alternate between blue and red edges because every column node in the cycle is incident to precisely one blue edge and one red edge, and every row node is incident to exactly one blue edge. Then $F_0 \diffsym C$ constitutes a perfect matching in $G$ with at least one red edge, and thus $\alpha(F_0 \diffsym C)$ is a solution of~\eqref{TNECP} by Lemma~\ref{lemma:sol_TNECP}. The graph $G$, the cycle $C$ and the point $\alpha(F_0 \diffsym C)$ can all be computed in polynomial time.
\end{proof}

Exploiting further the properties of $G$, we can even provide a lower bound on the number of solutions of~\eqref{TNECP}.

\begin{theorem}\label{thm:count}
The number of solutions of~\eqref{TNECP} is at least $2^\kappa - 1$ where~$\kappa$ is the number of connected components of $G$, and this bound is tight when the instance is nondegenerate.	
\end{theorem}

\begin{proof}
As in the proof of Theorem~\ref{thm:tnecp}, we remove red edges from $G$ so that every column node is incident to precisely one red edge. We denote by $G'$ the graph arising this way. In $G'$, every column node is incident to precisely one red edge and one blue edge. Hence for any two distinct perfect matchings of $G'$, there is a column node that is incident to a blue edge in one perfect matching and to a red edge in the other. Condition~\ref{iit} implies then that any two distinct perfect matchings $F$ of $G'$ provide distinct points $\alpha(F)$. For the perfect matchings $F \neq F_0$, this provides distinct solutions of~\eqref{TNECP} by Lemma~\ref{lemma:sol_TNECP}. Since $G'$ has at least as many connected components as $G$, we get the lower bound $2^\kappa - 1$ on the number of solutions of~\eqref{TNECP} once we know that there are at least $2^{\kappa'}-1$ perfect matchings in $G'$ in addition to the perfect matching $F_0$ formed by the blue edges. (Here, $\kappa'$ is the number of connected components in $G'$.) We prove now this lower bound on the number of perfect matchings in $G'$.

As shown in the proof of Theorem~\ref{thm:tnecp}, the edges of any cycle of $G' = (V,E')$ alternate between blue and red edges.
This implies that the symmetric difference of $F_0$ with any collection of node-disjoint cycles is again a perfect matching. Let $\mathcal{C}$ be the set of collections of node-disjoint cycles. The function $\phi \colon C \mapsto F_0 \diffsym C$ is an injective map from $\mathcal{C}$ to the set of perfect matchings. To get the desired lower bound on the number of perfect matchings in $G'$, it is thus enough to compute the cardinality of $\mathcal{C}$. Actually, every connected component of $G'$ contains only one cycle: indeed, orienting the blues edges from row to column nodes and the red edges the other way around, we get a graph whose vertices all have their outdegree equal to one; each connected component of a graph admitting such an orientation contains exactly one cycle. The cardinality of $\mathcal{C}$ is therefore equal to $2^{\kappa'}$, which provides the same lower bound on the number of perfect matchings of $G'$, as desired.

The symmetric difference of any two perfect matchings is an element of $\mathcal{C}$. Therefore, the map $\phi$ is surjective, and thus a bijection between $\mathcal{C}$ and the set of perfect matchings in $G'$. Since $G' = G$ in the nondegenerate case, we get that the number of perfect matchings in $G$ distinct from $F_0$ is exactly $2^{\kappa}-1$. We conclude by remarking that the function $\alpha$ bijectively maps perfect matchings distinct from $F_0$ to the solutions of~\eqref{TNECP}, by Lemma~\ref{lemma:sol_TNECP} and the arguments of the first paragraph.
\end{proof}

\section{Tropical bases and their classical realization}\label{sec:trop_bases}

The main purpose of this section is to prove Theorem~\ref{th:support}. To achieve this, we define the notion of tropical bases for tropical problems of the form $A \tdot x = b , \, x \in \T^d$. These latter problems are analogues of linear programming feasibility problems of the form $A x = b, x \in \R_{\geq 0}^d$, where $A$ and $b$ have nonnegative entries (recall that any element of $\T$ is implicitly ``nonnegative'' in the tropical sense). In comparison with the tropical bases introduced in~\cite{AllamigeonSIDMA15}, our notion of tropical bases applies to systems with nonnegative entries only, but that can be degenerate. We also define the dominance condition used in the statement of Theorem~\ref{th:support}, and show that, under this condition, the classical bases of linear programming feasibility problems are the same as the tropical bases of an associated tropical problem. We remark that tropical feasibility problems of the form $A \tdot x = b , \, x \in \T^d$ have been studied since the early age of tropical algebra; see~\eg~\cite{CuninghameGreen79,Zimmermann1981}. 
In particular, it can be decided in linear time if a system $A \tdot x = b, \, x \in \T^d$ has a solution. We emphasize that, in this section, we are interested in the structure of the solution sets of such problems (in terms of tropical bases) in order to handle complementarity constraints like in~\eqref{TNECP}.

\subsection{Tropical bases of nonnegative systems}\label{subsec:trop_bases}

We consider a tropical system $A \tdot x = b, \, x \in \T^d$, where $A \in \T^{n \times d}$ and $b \in \T^n$. 
A \emph{tropical (feasible) basis} is a subset $B \subset [d]$ of cardinality $n$ such that there exists a bijective map $\phi$ from $[n]$ to $B$ satisfying the following condition: for each $i \in [n]$, the quantity $b_i - A_{i \phi(i)}$ is in $\T$ and minimal among the terms $b_k - A_{k \phi(i)}$ for $k \in [n]$ (with the convention $-\infty - (-\infty) = +\infty$). 
A \emph{basic solution} associated with a basis $B$ is a feasible solution whose support (\ie, the index set of entries distinct from $-\infty$) is included in $B$. To every basis is associated at least one basic solution, namely the solution defined by $x_{\phi(i)} \coloneqq \min_k b_k - A_{k \phi(i)}$ for $i\in [n]$ and $x_j \coloneqq -\infty$ for $j \notin \phi([n])$. Unlike the classical case, there may be several basic solutions associated to a basis.

A tropical basis $B$ is \emph{nondegenerate} if for each $i \in [n]$, the minimum $\min_k (b_k - A_{k \phi(i)})$ belongs to $\R$ and is uniquely attained. Equivalently, the basis $B$ is nondegenerate if and only if every feasible solution whose support is included in $B$ has actually $B$ as support. In case $B$ is nondegenerate, there is a unique basic solution.  By extension, the instance $(A,b)$ is \emph{nondegenerate} if all tropical bases are nondegenerate. When there exists a basis, Lemma~\ref{lemma:trop_pivot} shows that every $j \in [d]$ such that the $j$th column of $A$ is not the $-\infty$ vector is contained in a basis. Therefore, when there is a basis, the instance $(A, b)$ is nondegenerate if and only if for each $j \in [d]$ such that the $j$th column of $A$ is not the $-\infty$ vector, the minimum $\min_k (b_k - A_{kj})$ belongs to $\R$ and is uniquely attained. In particular, under Condition~\ref{it}, the nondegeneracy of the instance $(A, q)$ with $A=\begin{pmatrix}
I & M \end{pmatrix}$ is equivalent to the nondegeneracy of the instance $(M, q)$ as defined in Section~\ref{sec:tnecp}, \ie, for each $j \in [n]$, the minimum $\min_i (q_i - M_{ij})$ is uniquely attained.

Deciding the existence of a basis reduces to deciding the existence of a matching of cardinality $n$ in a bipartite graph. It can thus be done in polynomial time. With the characterization above (uniqueness of the minimum for each $j$), nondegeneracy can also be decided in polynomial time.

\begin{remark}
In Section~\ref{sec:tnecp}, we have introduced a bipartite graph associated with~\eqref{TNECP}, whose edges are colored in blue and red, and proved that a certain point $\alpha(F)$ is a feasible solution when $F$ is a perfect matching with at least one red edge. It is not difficult to see that this perfect matching is closely related to the aforementioned matching used to decide the existence of a basis, and that $\alpha(F)$ is actually a basic solution of the system $w \tplus M \tdot z = q, \, (w,z) \in \T^n \times \T^n$.
\end{remark}

\begin{lemma}\label{lemma:trop_pivot}
Let $B$ be a tropical basis, and $j \notin B$. If the $j$th column of $A$ has at least one entry distinct from $-\infty$, then there exists a basis $B'$ included in $B \cup \{j\}$ and distinct from $B$. Moreover, this basis is unique if the instance $(A, b)$ is nondegenerate. 
\end{lemma}

The first part of the lemma ensures that a pivot operation can be defined over (degenerate) bases, even if several feasible solutions may have the same basis as support. This is consistent with the combinatorial abstractions of pivoting operations, especially the one developed by~\cite{Todd} in the context of complementarity problems. 

\begin{proof}[Proof of Lemma~\ref{lemma:trop_pivot}]
Let $\phi$ be a bijective map from $[n]$ to $B$ associated with the basis $B$. The minimum $\min_k (b_k - A_{kj})$ is in $\T$ since at least one $A_{kj}$ is distinct from $-\infty$. We take $i \in [n]$ such that $b_i - A_{ij} \in \arg\min_k (b_k - A_{kj})$, and we consider $B' \coloneqq B \setminus \{\phi(i)\} \cup \{j\}$. Then $B'$ is a basis associated with the bijective mapping $\phi'(k) = \phi(k)$ if $k \neq i$, and $\phi'(i) = j$. 

We now assume that $(A, b)$ is nondegenerate. Since $B$ and $B'$ are both nondegenerate, the minima $\min_k (b_k - A_{k \phi(i)})$ and $\min_k (b_k - A_{k \phi'(i)})$ are in $\R$, and both are uniquely attained for $k = i$. Thus, a basis included in $B \cup \{\phi'(i)\} = B' \cup \{\phi(i)\}$ cannot simultaneously contain $\phi(i)$ and $\phi'(i)$. We deduce that $B'$ is the unique basis contained $B \cup \{j\}$ and distinct from~$B$.
\end{proof}

\subsection{Classical realization of nondegenerate tropical systems}\label{subsec:realization}

We recall that a \emph{basis} of a classical system $\cla A x = \cla b, \, x \in \R_{\geq 0}^d$ is a subset $B \subset [d]$ such that the submatrix $\cla A_B$ formed by the columns of $\cla A$ indexed the elements in $B$ is a nonsingular square matrix. The \emph{basic solution} associated with $B$ is defined as the unique point satisfying $\cla A x = \cla b$ and whose support is included in $B$. The basis $B$ is \emph{feasible} if the basic solution is  nonnegative, and \emph{nondegenerate} if the support of the basic solution is equal to $B$. Note that a system $\cla A x = \cla b, \, x \in \R_{\geq 0}^d$ admits a basis if and only if the matrix $\cla A$ has full row rank.

We introduce the dominance condition. It makes use of the following terminology. A matrix $\cla M$ \emph{covers} a matrix $\cla M'$ if $\cla M - \cla M'$ is a nonnegative matrix. A matrix is \emph{columnwise normal} if its entries are in $[0,1]$ and every column has at least one $1$-entry. A columnwise normal matrix with $n$ rows satisfies the \emph{dominance condition} if the following two properties hold:
\begin{enumerate}[label=(\alph*)]
\item\label{item:a} there exists an $n \times n$ submatrix covering a permutation matrix;
\item\label{item:b} for any $n \times n$ submatrix covering a permutation matrix, the sum of each row is less than~$2$.
\end{enumerate}
Consider a system $\cla A x = \cla b, \, x \in \R_{\geq 0}^d$ such that $\cla A$ is a nonnegative matrix with a nonzero entry in every column, and $\cla b$ is a positive vector. The \emph{normalized matrix} $\norm{A}$ is defined as $(\diag {\cla b})^{-1} \cla A (\diag \cla u)^{-1}$, where $\cla u_j$ is the largest entry of the $j$th column of $(\diag {\cla b})^{-1} \cla A$. In this way, the matrix $\norm{A}$ is columnwise normal. Remark that since $\cla b_i > 0$ and $\cla u_j > 0$ for all $i,j$, normalizing the system $\cla A x = \cla b, \, x \in \R_{\geq 0}^d$ into the system $\norm{A} x = e , \, x \in \R_{\geq 0}^d$, where $e$ is the all-$1$ vector, preserves the feasible bases and the supports of the basic feasible solutions. By abuse of language, we say that the system $\cla A x = \cla b, \, x \in \R_{\geq 0}^d$ satisfies the \emph{dominance condition} if $\norm{A}$ does. 

\begin{example}
Consider the following columnwise normal matrix \[
\begin{bmatrix}
1 & 0.5 & 0.4 & 1 \\
0 & 1 & 0.2 & 0.6 \\
0.6 & 0.3 & 1 & 0.5
\end{bmatrix} 
\]
There are exactly two $3 \times 3$-submatrices that cover a permutation matrix:
\[
\begin{bmatrix}
1 & 0.5 & 0.4 \\
0 & 1 & 0.2 \\
0.6 & 0.3 & 1 
\end{bmatrix} 
\qquad \qquad 
\begin{bmatrix}
0.5 & 0.4 & 1\\
1 & 0.2 & 0.6 \\
0.3 & 1 & 0.5
\end{bmatrix}
\]
In both, it can be checked that the sum of each row is less than $2$. The considered $3\times 4$-matrix satisfies therefore the dominance condition.
\end{example}

The following result establishes the correspondence between the feasible bases of classical and tropical systems under the dominance condition. As we shall see, it plays a key role in the proof of Theorem~\ref{th:support}.

\begin{proposition}\label{prop:classical_basis}
Let $\cla A \in \R^{n \times d}_{\geq 0}$ and $\cla b \in \R^n_{\geq 0}$ such that no column of $\cla A$ is the $0$ vector and no entry of $\cla b$ is equal to $0$. Assume that that the (classical) system $\cla A x = \cla b, \, x \in \R_{\geq 0}^d$ satisfies the dominance condition. Let $\tro A \in \T^{n \times d}$ and $\tro b \in \T^n$ be such that no column of $\tro A$ is the $-\infty$ vector and no entry of $\tro b$ is equal to $-\infty$.

If the minima $\min_i (\cla b_i / \cla A_{ij})$ and $\min_i (\tro b_i - \tro A_{ij})$ are reached by the same indices (with the convention $\lambda/0 = +\infty$ for all $\lambda > 0$), then the feasible bases of the classical system are the same as the tropical bases of the tropical system $\tro A \tdot x = \tro b, \, x \in \T^d$, and both systems are nondegenerate.
\end{proposition}

We start with a technical lemma that will be useful to characterize feasible bases.
\begin{lemma}\label{lemma:sub_sto}
Let $F$ be a nonnegative square matrix such that $\sum_j F_{ij} < 1$ for all $i$. Then the matrix $I+F$ is nonsingular, and the vector $(I+F)^{-1} e$ has positive entries.
\end{lemma}

\begin{proof}
The matrix $I+F$ is nonsingular because the norm $\|F\| \coloneqq \sup_{x \neq 0} \frac{\| F x \|_\infty}{\| x \|_\infty} = \max_i \sum_j | F_{ij}|$ of the matrix $F$ is strictly less than $1$. Similarly, $\|F^2\| \leq \| F\|^2 < 1$, thus $I-F^2$ is nonsingular, with inverse $\sum_{i = 0}^\infty F^{2i}$. Since $I - F^2 = (I - F) (I + F)$, we have
\[
(I+F)^{-1} e = (I-F^2)^{-1} (I - F) e = \Bigl(I + \sum_{i = 1}^\infty F^{2i}\Bigr) (I - F) e \, .
\]
The entries of $(I - F) e$ are of the form $1 - \sum_j F_{ij} > 0$. As $F$ is a nonnegative matrix, we have $(I + F)^{-1} e \geq (I - F) e$. We deduce that $(I + F)^{-1} e$ has positive entries.
\end{proof}

The next two lemmas are related with matrices satisfying the dominance condition.
\begin{lemma}\label{lemma:dom1}
A matrix that satisfies the dominance condition has full row rank, and every column contains precisely one $1$-entry.
\end{lemma}

\begin{proof}
Let $C$ be a matrix that satisfies the dominance condition. We pick a submatrix $C_0$ of~$C$ that covers a permutation matrix. Up to permuting its rows, $C_0$ writes as $I + F$ where~$F$ satisfies the condition of Lemma~\ref{lemma:sub_sto}. Thus, $C_0$ is nonsingular, which proves that $C$ has full row rank.

Suppose now that the $j$th column of $C$ contains at least two $1$-entries. By exchanging one column of $C_0$ with this column, we can build a submatrix of $C$ that covers a permutation matrix and in which there is a row containing at least two $1$-entries. This contradicts the dominance condition.
\end{proof}

\begin{lemma}\label{lemma:dom2}
Consider a matrix $C$ that satisfies the dominance condition, and a submatrix $C'$ of $C$. Suppose that $C'$ has one $1$-entry in every column and at most one $1$-entry in each row. Then, the sum of the non-$1$-entries in each row of $C'$ is less than $1$.
\end{lemma}

\begin{proof}
As in the proof of Lemma~\ref{lemma:dom1}, $C_0$ refers to a submatrix of $C$ that covers a permutation matrix.

Suppose that $C'$ is the submatrix $C_{I, J} = (C_{ij})_{(i,j) \in I \times J}$. Let $n$ be the number of rows of $C$. Lemma~\ref{lemma:dom1} shows that the submatrix $C_{[n], J}$ satisfies the same conditions as $C'$, \ie, it has one $1$-entry in every column and at most one $1$-entry in each row. Besides, the cardinality of $J$ is less than or equal to $n$. Therefore, we can complete $C_{[n], J}$ with columns of $C_0$ in such a way we get an $n \times n$ submatrix of $C$ that covers a permutation matrix. The expected result follows from the fact that the sum of each row of this submatrix is less than $2$, thanks to Item~\ref{item:b} of the dominance condition.
\end{proof}

Given a matrix $C$, we denote by $C_J$ the submatrix formed by the columns of $C$ indexed by~$J$.
\begin{lemma}\label{lemma:dom3}
For every feasible basis $B$ of a system $\cla A x = \cla b, \, x \in \R_{\geq 0}^d$ that satisfies the dominance condition, the submatrix $\norm{A}_{B}$ covers a permutation matrix.
\end{lemma}

\begin{proof}
It suffices to prove the statement for systems of the form $C x = e, \, x \in \R_{\geq 0}^d$ where $C$ satisfies the dominance condition. We prove it by induction on the number $n$ of rows of $C$.

The case $n = 1$ is trivial: since $C$ is columnwise normal, it is reduced to the all-$1$ row. We now suppose $n > 1$. Let $B$ be a feasible basis. Let $I \subset [n]$ be the set of rows $i$ such that there is $j \in B$ with $C_{ij} = 1$. For contradiction, suppose that $I \subsetneq [n]$, and let $i_0 \in [n] \setminus I$. We claim that the matrix $C_{I, B} = (C_{ij})_{(i,j) \in I \times B}$ satisfies the dominance condition. By construction, this matrix is columnwise normal. Besides, by Lemma~\ref{lemma:dom1}, every column of $C$ in $B$ has exactly one $1$-entry, and this entry is necessarily in the rows in $I$. Therefore, we can select for each row $i \in I$ a $1$-entry from a distinct column $j \in B$. These columns form a submatrix of $C_{I,B}$ that covers a  permutation matrix, which proves Item~\ref{item:a}. Now take any submatrix $C'$ of $C_{I, B}$ that covers a permutation matrix. It satisfies the condition of Lemma~\ref{lemma:dom2}. Thus, the sum of each of its rows is less than $2$, which proves Item~\ref{item:b}. This proves the claim.

We now consider the following problem:
\begin{equation}\label{eq:P}
\begin{array}{r@{\quad}l}
\text{Maximize} & {\displaystyle\sum_{j \in B} C_{i_0 j} x_j} \\[\jot]
\text{subject to} & C_{I, B} x = e_I \,  \\[\jot]
& x \in \R_{\geq 0}^B \, .
\end{array}
\tag{P}
\end{equation}
Problem~\eqref{eq:P} is feasible (as $B$ is a feasible basis of the original system) and bounded (each column of $C_{I,B}$ contains a $1$-entry). Thus, it admits an optimal basis $B^* \subset B$. By induction hypothesis, the submatrix $C_{I,B^*}$ covers a permutation matrix. Since every column of $C_{I, B^*}$ contains exactly one $1$-entry, we deduce that the same applies to every row of $C_{I, B^*}$. Thus, $C_{[n], B^*}$ satisfies the condition of $C'$ of Lemma~\ref{lemma:dom2}. Denoting by $x^*$ the basic solution of~\eqref{eq:P} associated with basis $B^*$, we have:
\[
\sum_{j \in B} C_{i_0 j} x^*_j = \sum_{j \in B^*} C_{i_0 j} x^*_j \leq \sum_{j \in B^*} C_{i_0 j} < 1 \, ,
\] 
where the first inequality holds because the entries of $x^*$ are less than or equal to $1$. (Indeed, every column of $C_{I, B^*}$ contains exactly one $1$-entry.) 

As a consequence, there is no point $x$ satisfying $C x = e, \, x \geq 0$ with support included in $B$. This contradicts the fact that $B$ is a feasible basis. Therefore, $I = [n]$, which means that every row of $C_B$ contains at least one $1$-entry. By Lemma~\ref{lemma:dom1}, every column of $C_B$ contains exactly one $1$-entry. Hence, $C_B$ covers a permutation matrix.
\end{proof}

\begin{proof}[\proofname{} of Proposition~\ref{prop:classical_basis}]
By assumption, for all $j \in [d]$, we have
\begin{equation}\label{eq:arg}
\arg\min_i (\tro b_i - \tro A_{ij}) = \arg\min_i (\cla b_i / \cla A_{ij}) = \arg\max_i \norm A_{ij} \, ,
\end{equation}
and every minimum is finite. 

Let $B$ be a tropical basis of $\tro A \tdot x = \tro b, \, x \in \T^d$, and $\phi$ a bijective mapping associated with $B$. By~\eqref{eq:arg}, the equality $\norm A_{i \phi(i)} = 1$ holds for all $i \in [n]$. Hence, there is a permutation matrix $P$ covered by $\norm A_B$. The matrix $P^{-1} \norm A_B$ is nonnegative and its diagonal elements are equal to $1$, thus it writes as $I + F$ where $F$ is nonnegative. Besides, the sum of each row of $F$ is less than $1$ since the sum of each row of $\norm A_B$ is less than $2$. Thus, $F$ satisfies the condition of Lemma~\ref{lemma:sub_sto}. In consequence, the system $\norm A_B x = e$ has a unique solution, and this solution has positive entries. We deduce that $B$ is a (nondegenerate) feasible basis of $\cla A x = \cla b, \, x \in \R_{\geq 0}^d$. 

Conversely, let $B$ be a feasible basis of $\cla A x = \cla b, \, x \in \R_{\geq 0}^d$. The permutation matrix covered by the submatrix $\norm A_B$, whose existence is ensured by Lemma~\ref{lemma:dom3}, provides 
a bijective mapping $\phi$ between $[n]$ and $B$ such that $\norm A_{i \phi(i)} = 1$ for all $i$. By~\eqref{eq:arg}, $B$ is a tropical basis of $\tro A \tdot x = \tro b, \, x \in \T^d$. This concludes the proof that the feasible bases of the classical system are the same as the tropical bases of the tropical system.

Let us now prove that both systems are nondegenerate. Every feasible basis of the classical system is nondegenerate because it is a tropical basis, and, as shown in the second paragraph above, any such basis is a nondegenerate feasible basis of the classical system. Every tropical basis of the tropical system is nondegenerate because every column of $\norm A$ contains exactly one $1$-entry (Lemma~\ref{lemma:dom1}) and the equalities~\eqref{eq:arg} show then that the minimum 
$\min_i (\tro b_i - \tro A_{ij})$ is uniquely attained.
\end{proof}

\begin{proposition}
Deciding whether $\cla A x = \cla b, \, x \in \R_{\geq 0}^d$ satisfies the dominance condition can be done in polynomial time.
\end{proposition}

\begin{proof}
We suppose that $\cla A$ is an $n \times d$ nonnegative matrix with a nonzero entry in every column, and $\cla b$ is a positive vector. The normalized matrix $\norm A$ is a columnwise normal matrix that can be computed in polynomial time.

The first step of the algorithm is to check that every column of $\norm A$ contains precisely one entry equal to $1$. This is a necessary condition as shown by Lemma~\ref{lemma:dom1}. In the rest of the proof, we assume that this property is satisfied.

For each $i \in [n]$, we introduce the set $C_i \subset [d]$ consisting of the columns with a $1$-entry on the $i$th row. The sets $C_i$ are pairwise disjoint and cover $[d]$. The second step of the algorithm consists in checking that:
\begin{equation}
\forall i \in [n] \, , \quad C_i \neq \varnothing \quad \text{and} \quad \sum_{k = 1}^n \max_{j \in C_k} \norm A_{ij} < 2 \, . \label{eq:dominant}
\end{equation}
We claim that the dominance condition is satisfied if and only if~\eqref{eq:dominant} holds. Since it can be checked in polynomial time, this will complete the proof.

Suppose that the dominance condition is satisfied. Item~\ref{item:a} ensures that every $C_i$ is nonempty. Let $i \in [n]$. We introduce $j_k \in \arg\max_{j \in C_k} \norm A_{ij}$ for each $k \in [n]$. We consider the submatrix $\cla A'$ formed by the columns of $\norm A$ indexed by the $j_k$. Every row and every column of $\cla A'$ contain precisely one entry equal to $1$. Hence there is a (unique) permutation matrix $P$ covered by $\cla A'$. The inequality in~\eqref{eq:dominant} is satisfied by Item~\ref{item:b} because its left-hand side is equal to the sum of the $i$th row of $\cla A'$.

Conversely, suppose that~\eqref{eq:dominant} is satisfied. Item~\ref{item:a} is verified by taking a submatrix formed by one column in every $C_i$. 
We now check Item~\ref{item:b}. Let $j_1, \dots, j_n$ pairwise distinct elements of $[d]$. We introduce the $n \times n$ matrix $\cla A'$ formed by the columns of $\norm A$ indexed by the $j_k$, and we suppose that $\cla A'$ covers  some permutation matrix $P$. Let $\pi \colon [n] \to \{j_1, \dots, j_n\}$ be the bijective mapping induced by $P$, \ie, $P_{k l} = 1$ if $j_l = \pi(k)$. For each $k \in [n]$, we have $\pi(k) \in C_k$. Thus, for all $i \in [n]$,
\[
\sum_{j = 1}^n \cla A'_{i j} = \sum_{k = 1}^n \norm A_{i j_k} = \sum_{k = 1}^n \norm A_{i \pi(k)} \leq \sum_{k = 1}^n \max_{j \in C_k} \norm A_{ij} < 2 \; . \qedhere
\] 
\end{proof}

We end this section by proving Theorem~\ref{th:support}.

\begin{proof}[\proofname{} of Theorem~\ref{th:support}]
Let $\cla A \coloneqq \begin{pmatrix} I & -M \end{pmatrix}$ and $\cla b \coloneqq \cla q$. We define $\tro A \coloneqq \log \cla A$ and $\tro b \coloneqq \log \cla b$. (With this notation, the instance of~\eqref{TNECP} considered here is $\tro A \tdot \begin{psmallmatrix} w \\ z \end{psmallmatrix} = \tro b, \, \trans{w} \tdot z = -\infty$.) Thanks to the conditions~\ref{i} and~\ref{ii} (resp.~\ref{it} and~\ref{iit}) as well as the fact that the instance of~\eqref{NECP} satisfies the dominance condition, we can apply Proposition~\ref{prop:classical_basis}. This ensures that the two systems $\cla A x = \cla b, \, x \in \R_{\geq 0}^n$ and $\tro A \tdot x = \tro b, \, x \in \T^n$ have the same feasible bases and are nondegenerate. The conclusion will follow from the fact that, for both systems, all feasible solutions are basic and the support of any feasible solution is a basis. This is what we prove now.

Because of the constraints of the form $\trans{w} z = 0$, the cardinality of the support of any solution of~\eqref{NECP} is at most $n$, where $n$ is the number of rows of $M$. The system being nondegenerate, the support is actually of cardinality $n$ and is a basis.

Similarly, the constraint of the form $\trans{w} \tdot z=-\infty$ implies that the cardinality of the support of any solution of~\eqref{TNECP}  is at most $n$. As noted in Section~\ref{subsec:realization}, the nondegeneracy of the system makes that the minimum $\min_k (\tro b_k - \tro A_{kj})$ is uniquely attained. It implies that the support of any solution is of cardinality $n$ (each column in the support contributes to the satisfaction of exactly one row of the system) and shows the existence of a bijective map $\phi$ ensuring that any such support is a basis. 
\end{proof}

\section{The Lemke--Howson algorithm for the tropical Nash equilibrium complementarity problem}\label{sec:trop-LH}

\subsection{Correctness of the algorithm in the tropical setting}\label{subsec:trop_LH}

The purpose of this section is to show that the Lemke--Howson algorithm handles tropical Nash equilibrium complementarity problems, up to replacing the classical notion of bases by the tropical one.

We define the \emph{disjoint union} of two sets $S, S'$, denoted by $S \uplus S'$, as the set $(S \times \{\text{blue}\}) \cup (S' \times \{\text{red}\})$. The \emph{label} of an element $(k,\text{blue})$ or $(k,\text{red})$ of $[n] \uplus [n]$ is its first component~$k$. Two elements of $[n] \uplus [n]$ are \emph{twins} if they have the same label but distinct colors. 

The Lemke--Howson algorithm usually takes as input a nondegenerate system $w = \cla M z + \cla q, \, (w, z) \in \R^{n+n}_{\geq 0}$. The latter system writes as 
$\begin{pmatrix}
I & -M 	
\end{pmatrix} \begin{psmallmatrix} w \\ z \end{psmallmatrix} = q$, so that its bases are understood as subsets of $[n] \uplus [n]$. The algorithm first fixes an arbitrary integer $j^\star \in [n]$, and makes use of the notion of fully labeled and almost fully labeled bases. A basis $B$ is \emph{fully labeled} if all possible labels $j \in [n]$ appear in $B$. The basis $B$ is \emph{almost fully labeled} if all possible labels but one appear in $B$, and the label $j^\star$ appears with the two colors. 

The Lemke--Howson algorithm is described in Algorithm~\ref{fig:tropLH}. It performs a number of pivot operations, thus generating a sequence of feasible bases starting from $[n] \uplus \varnothing$ (which is a fully labeled feasible basis corresponding to the trivial solution $(w, z) = (q, 0)$). We observe that each basis is either fully labeled (in the first and last iterations) or almost fully labeled (in the intermediate iterations). The operation at Line~\lineref{line:pivot} is well-defined thanks to the classical analogue of Lemma~\ref{lemma:trop_pivot}: under the nondegeneracy assumption and the condition~\ref{i}, for any (classical) feasible basis $B$ and $j \notin B$, there is a unique feasible basis included in $B \cup \{j\}$ and distinct from $B$. This ensures that the algorithm terminates with a fully labeled feasible basis distinct from $[n] \uplus \varnothing$; see, e.g.,~\cite{von2002computing} for the classical proof of this fact (which we actually reproduce below in the tropical setting). The corresponding basic solution is a solution of~\eqref{NECP}. 

Similarly, we consider a nondegenerate tropical system $w \tplus M^- \tdot z = q^+, \, (w, z) \in \T^{n+n}$. As above, we can write this system under the form $A \tdot \begin{psmallmatrix} w \\ z \end{psmallmatrix} = q^+$, where we index the columns of $A$ by elements of $[n] \uplus [n]$. As noted in Section~\ref{subsec:trop_bases}, the nondegeneracy of the instance $(A, q^+)$ is equivalent to the nondegeneracy of the instance $(M^-, q^+)$ as defined in Section~\ref{sec:tnecp}, \ie, for each $j \in [n]$, the minimum $\min_i (q^+_i - M^-_{ij})$ is uniquely attained. 

Observe that the notion of (almost) fully labeled still makes sense for tropical bases. Besides, the set $[n] \uplus \varnothing$ is a (fully labeled) feasible basis of the system $w \tplus M^- \tdot z = q^+, \, (w, z) \in \T^{n+n}$. Finally, the operation done at Line~\lineref{line:pivot} in the Lemke--Howson algorithm is valid thanks to Lemma~\ref{lemma:trop_pivot} (the hypotheses of this lemma are satisfied thanks to the nondegeneracy assumption and the condition~\ref{it}). As a consequence, the Lemke--Howson algorithm applies to the tropical system $w \tplus M^- \tdot z = q^+, \, (w, z) \in \T^{n+n}$. As for classical systems, it iterates over almost fully labeled bases until it finds a fully labeled basis distinct from $[n] \uplus \varnothing$.

\begin{algorithm}
\begin{center}
\begin{algorithmic}[1]
\State $B \leftarrow [n] \uplus \varnothing$
\State $\gamma \leftarrow$ $(j^\star, \text{red})$
\Loop
	\State $B' \leftarrow$ unique basis included in $B \cup \{\gamma\}$ and distinct from $B$\label{line:pivot}
    \State\label{line:twin} $\gamma \leftarrow$ twin of the unique element in $B \setminus B'$
    \State $B \leftarrow B'$
        \sIf{$B$ is fully labeled} STOP
\EndLoop
\end{algorithmic}
\end{center}	
\caption{The Lemke--Howson algorithm} \label{fig:tropLH}
\end{algorithm}

The key property to prove the correctness and termination of the Lemke--Howson algorithm in the tropical case is Lemma~\ref{lemma:trop_pivot}. While the proof is the same as usual, we provide it for the sake of completeness. 

\begin{proposition}\label{prop:term-LH}
On nondegenerate~\eqref{TNECP} instances, the Lemke--Howson algorithm terminates after a finite number of iterations with a fully labeled basis $B$ distinct from $[n] \uplus \varnothing$.
\end{proposition}

The basic point $(w,z)$ associated with a basis $B$ is given by
\[
w_i = \begin{cases}
q^+_i & \text{if} \; (i, \text{blue}) \in B \, , \\
-\infty & \text{otherwise,}
\end{cases}
\qquad \qquad 
z_j = \begin{cases}
 \min_i (q^+_i - M^-_{ij}) & \text{if} \; (j, \text{red}) \in B \, , \\
-\infty & \text{otherwise.}
\end{cases} 
\]
When $B$ is fully labeled and distinct from $[n] \uplus \varnothing$, we have further $\trans{w} \tdot z=-\infty$ and $z \neq - \infty$. Therefore, the basic point associated with a fully labeled basis $B$ distinct from $[n] \uplus \varnothing$ is a solution of~\eqref{TNECP}. 

We prove now Proposition~\ref{prop:term-LH}. Two bases $B$ and $B'$ are \emph{adjacent} if $B \diffsym B'$ has cardinality two, all labels appear in $B \cup B'$, and the label $j^\star$ appears in $B \cup B'$ with the two colors. Note that the Lemke--Howson algorithm always moves from a basis to an adjacent one.

\begin{lemma}\label{lemma:lh2}
The fully labeled basis $[n] \uplus \varnothing$ is adjacent to exactly one basis.
\end{lemma}

\begin{proof}
Let $B \coloneqq [n] \uplus \varnothing$, and $B'$ an adjacent basis. The latter necessarily contains $(j^\star, \text{red})$. By Lemma~\ref{lemma:trop_pivot}, it is the unique basis included in $B \cup \{(j^\star, \text{red})\}$ other than $B$.
\end{proof}

\begin{lemma}\label{lemma:lh1}
Every almost fully labeled basis is adjacent to exactly two bases.
\end{lemma}

\begin{proof}
Let $B$ be an almost fully labeled basis. Let $j$ be the missing label in $B$. This label must appear in any adjacent basis to $B$. Thus, any such basis is contained either in $B \cup \{(j, \text{blue})\}$ or $B \cup \{(j, \text{red})\}$. By Lemma~\ref{lemma:trop_pivot} and the nondegeneracy assumption (which applies thanks to the condition~\ref{it}), there are exactly two such bases not equal to $B$.
\end{proof}

\begin{proof}[Proof of Proposition~\ref{prop:term-LH}]
Assume for a contradiction that the algorithm visits a basis at least twice. Let $B'$ be the basis whose second occurrence along the execution of the algorithm is the earliest, and let $B$ be the basis from which $B'$ is reached for the second time. The bases $B$ and $B'$ are adjacent. Since $B$ has only been visited once yet, Lemma~\ref{lemma:lh2} shows that $B'$ is not the fully labeled basis $[n] \uplus \varnothing$. It is not another fully labeled basis since otherwise the algorithm would have stopped at its first occurrence. Let now $B''$ be the basis visited by the algorithm after the second visit of $B'$. The bases $B'$ and $B''$ are adjacent. Besides, $B$ and $B''$ are distinct, because $B''$ contains the twin of an element of $B$, and this element is taken form $B \setminus B'$ (as $B'$ is almost fully labeled, the label of this element cannot be $j^\star$). Since $B'$ is an almost fully labeled basis, we get a contradiction with Lemma~\ref{lemma:lh1}: in addition to $B$ and $B''$, the basis from which $B'$ was reached for the first time is a third basis adjacent to $B'$.

Since every basis is visited at most once, the algorithm terminates. Furthermore, the final basis is fully labeled and distinct from the initial one, otherwise this latter would be visited twice.
\end{proof}

\begin{remark}\label{remark:nondegenerate}
As in the classical setting, the Lemke--Howson algorithm applies to nondegenerate instances of~\eqref{TNECP}. The general case of~\eqref{TNECP} can be reduced to this case by perturbing the entries of each column of the matrix $M^-$ so that the minimum $\min_i (q^+_i - M^-_{ij})$ is attained by a unique term for each $j \in [n]$. We point out that this perturbation can be done symbolically rather than numerically, \ie, we can (arbitrarily) select one term $q^+_i - M^-_{ij}$ among the ones reaching the minimum for each $j$. This precisely amounts to removing extra red edges in the graph $G$ introduced in Section~\ref{subsec:tncep}, as we do in the first part of the proof of Theorem~\ref{thm:tnecp}.
\end{remark}

\subsection{Complexity results}\label{subsec:classical_LH}

First we prove Theorem~\ref{th:complexity} by providing a linear bound on the length of the sequences of bases visited by the Lemke--Howson algorithm when applied to a nondegenerate instance of~\eqref{TNECP}.

\begin{proof}[\proofname\ of Theorem~\ref{th:complexity}]
In order to match the notation of Section~\ref{subsec:trop_bases}, we 
write the system $w \tplus M^- \tdot z = q^+$ as $A \tdot \begin{psmallmatrix} w \\ z \end{psmallmatrix} = b$. Recall that the instance $(A, b)$ is nondegenerate. 

Let $B_1, \dots, B_p$ be the sequence of bases visited by the  Lemke--Howson algorithm, where $B_1 = [n] \uplus \varnothing$ and $B_p$ is another fully labeled basis. If $B$ and $B'$ are two successive bases, $B'$ is obtained by adding an element $\gamma \notin B$ to $B$ and removing an element $\delta \in B$. They are respectively the \emph{entering} and \emph{leaving} columns. Moreover, the minima $\min_i (b_i - A_{i\gamma})$ and $\min_i (b_i - A_{i\delta})$ are both uniquely attained at the same element $i \in [n]$. We denote by $i_1, \dots, i_{p-1}$ the sequence of these elements. 

By contradiction, we suppose $p > 2n-1$. We introduce the smallest integer $l$ such that $i_l = i_k$ for some $k < l \leq n$. We now prove the following claim by induction on $0 \leq s \leq k-1$:
\begin{enumerate}[label=(C\arabic*)]
\item\label{item:c} the basis $B_{l+s}$ contains the columns that have entered at iterations $1, \dots, k-s$; 
\item \label{item:d} the leaving column at iteration $l+s$ is the entering column at iteration $k-s$.
\end{enumerate}

We start with $s = 0$. Since $i_1, \dots, i_{l-1}$ are pairwise distinct, every column entered at an iteration less than $l$ is in the basis $B_l$. The leaving column at iteration $l$ is uniquely determined as the element $\gamma \in B_l$ such that the minimum $\min_i (b_i - A_{i\gamma})$ is attained at $i = i_l$. Since $i_l = i_k$, the column $\gamma$ is the one that has entered at iteration $k$. 

We suppose that the claim holds for $s \geq 0$. The basis $B_{l+s+1}$ contains the elements of $B_{l+s}$ except the column which has entered at iteration $k-s$ (thanks to Item~\ref{item:d} of the induction hypothesis). By Item~\ref{item:c} of the induction hypothesis, the basis $B_{l+s+1}$ contains all the columns entered at iterations $1, \dots, k-s-1$. This proves Item~\ref{item:c} at $s+1$. Moreover, the entering column at iteration $l+s+1$ is equal to the leaving column at iteration $k-s-1$, because both are the twins of the leaving column at iteration $l+s$ (which is also the entering column at iteration $k-s$ by Item~\ref{item:d} of the induction hypothesis); see Line~\lineref{line:twin} in Algorithm~\ref{fig:tropLH}. We deduce that $i_{l+s+1} = i_{k-s-1}$. The leaving column at iteration $l+s+1$ is uniquely determined as the element $\gamma \in B_{l+s+1}$ such that the minimum $\min_i (b_i - A_{i\gamma})$ is attained at $i = i_{k-s-1}$. This is precisely the column that has entered at iteration $k-s-1$, thanks to Item~\ref{item:c} at $s+1$. This completes the proof of the claim.

As a consequence, the leaving column at iteration $l+k-1$ is equal to $(j^\star, \text{red})$, so that the basis $B_{l+k}$ is fully labeled. Therefore, $p = l+k \leq 2n-1$, which is a contradiction.
\end{proof}

\begin{remark}
When applied to a nondegenerate instance of~\eqref{TNECP}, the Lemke--Howson algorithm actually implements a way to find a cycle in the graph $G$ involved in the proof of Theorem~\ref{thm:tnecp}. Indeed, in the case where $(M^-,q^+)$ is nondegenerate, every column node of the graph $G$ introduced in Section~\ref{subsec:tncep} is incident to exactly one blue edge and one red edge. In this setting, every basis $B \subset [n] \uplus [n]$ can be equivalently thought of as a subset of edges of the graph $G$ in which every row node has degree $1$; this subset consists of the blue edges $u_i v_i$ for all $(i, \text{blue}) \in B$, and the red edges incident to the nodes $v_j$ for all $(j, \text{red}) \in B$. In this way, fully labeled bases correspond to perfect matchings. The entering and leaving columns at every iteration of the Lemke--Howson algorithm correspond to entering and leaving edges incident to a same row node. Moreover, the sequence of such row nodes corresponds to the elements $i_1, \dots, i_{p-1}$ in the proof of Theorem~\ref{th:complexity}. We illustrate in Figure~\ref{fig:trace-bis} the two possible cases, depending on whether $i_1, \dots, i_{p-1}$ are pairwise distinct or not (the latter case happens when 
$p \geq n+2$). In both situations, we remark that the Lemke--Howson algorithm amounts to finding a cycle $C$ in the graph $G$ and taking the symmetric difference with the initial perfect matching, like in the proof of Theorem~\ref{thm:tnecp}.
\end{remark}

\begin{figure}
\begin{center}
\begin{tikzpicture}
\begin{scope}[shift={(2,0)}]
\node at (-5,0) {(a)};
\foreach \a in {0,1,...,9}{
\coordinate (v\a) at (\a*360/10: 1.5cm);
\coordinate (l\a) at (\a*360/10: 1.9cm);
}
\draw[red,ultra thick] (v1) -- (v2) (v3) -- (v4) (v5) -- (v6) (v7) -- (v8) (v9) -- (v0);
\draw[blue,very thick,dashed] (v0) -- (v1) (v2)  -- (v3) (v4) -- (v5) (v6) -- (v7) (v8) -- (v9);
\foreach \a in {0,1,...,9}{
\filldraw (v\a) circle (1pt);
}
\node[anchor=east,inner sep=-4pt] at (l4) {$v_{j^\star}$};
\node at (l3) {$u_{i_1}$};
\node at (l1) {$u_{i_2}$};
\node at (l9) {$u_{i_3}$};
\node at (l7) {$u_{i_4}$};
\node at (l5) {$u_{i_5}$};
\end{scope}
\begin{scope}[shift={(0,-4.5)}]
\node at (-5.5,0) {(b)};
\foreach \a in {0,1,...,9}{
\coordinate (v\a) at (\a*360/10: 1.5cm);
\coordinate (l\a) at (\a*360/10: 1.9cm);
}
\path (v5) ++ (-0.92,0) coordinate (w2) ++ (-0.92,0) coordinate (w1) ++ (-0.92,0) coordinate (w0);

\draw[red,ultra thick] (w0) -- (w1) (w2) -- (v5);
\draw[blue,very thick,dashed] (w1) -- (w2);
\draw[red,ultra thick] (v1) -- (v2) (v3) -- (v4) (v7) -- (v8) (v9) -- (v0);
\draw[blue,very thick,dashed] (v0) -- (v1) (v2)  -- (v3) (v4) -- (v5) (v6) -- (v7) (v8) -- (v9);

\foreach \a in {0,1,...,9}{
\filldraw (v\a) circle (1pt);
}
\filldraw (w0) circle (1pt);
\filldraw (w1) circle (1pt);
\filldraw (w2) circle (1pt);
\node[above=0.25cm,anchor=base] at (w0) {$v_{j^\star}$};
\node[right] at (v5) {$u_{i_k} = u_{i_l}$};
\end{scope}
\begin{scope}[shift={(7,-4.5)}]
\foreach \a in {0,1,...,9}{
\coordinate (v\a) at (\a*360/10: 1.5cm);
\coordinate (l\a) at (\a*360/10: 1.9cm);
}
\path (v5) ++ (-0.92,0) coordinate (w2) ++ (-0.92,0) coordinate (w1) ++ (-0.92,0) coordinate (w0);

\draw[blue,ultra thick] (w1) -- (w2);
\draw[red,very thick,dashed] (w0) -- (w1) (w2) -- (v5);
\draw[red,ultra thick] (v1) -- (v2) (v3) -- (v4) (v5) -- (v6) (v7) -- (v8) (v9) -- (v0);
\draw[blue,very thick,dashed] (v0) -- (v1) (v2)  -- (v3) (v4) -- (v5) (v6) -- (v7) (v8) -- (v9);

\foreach \a in {0,1,...,9}{
\filldraw (v\a) circle (1pt);
}
\filldraw (w0) circle (1pt);
\filldraw (w1) circle (1pt);
\filldraw (w2) circle (1pt);
\node[above=0.25cm,anchor=base] at (w0) {$v_{j^\star}$};
\node[right] at (v5) {$u_{i_k} = u_{i_l}$};
\end{scope}
\end{tikzpicture}	
\end{center}
\caption{An interpretation of the sequence of iterations of the Lemke--Howson algorithm on the graph $G$ of Section~\ref{subsec:tncep}. Solid and dashed line segments respectively correspond to added and removed edges. The top part (a) illustrates the case where the visited row nodes are pairwise distinct. The bottom part (b) is the case where some row nodes are visited twice. Using the notation of the proof of Theorem~\ref{th:complexity}, the left-hand side corresponds to the state just before cycling on the row node $u_{i_l} = u_{i_k}$, while the right-hand side provides the final state.}\label{fig:trace-bis}
\end{figure}
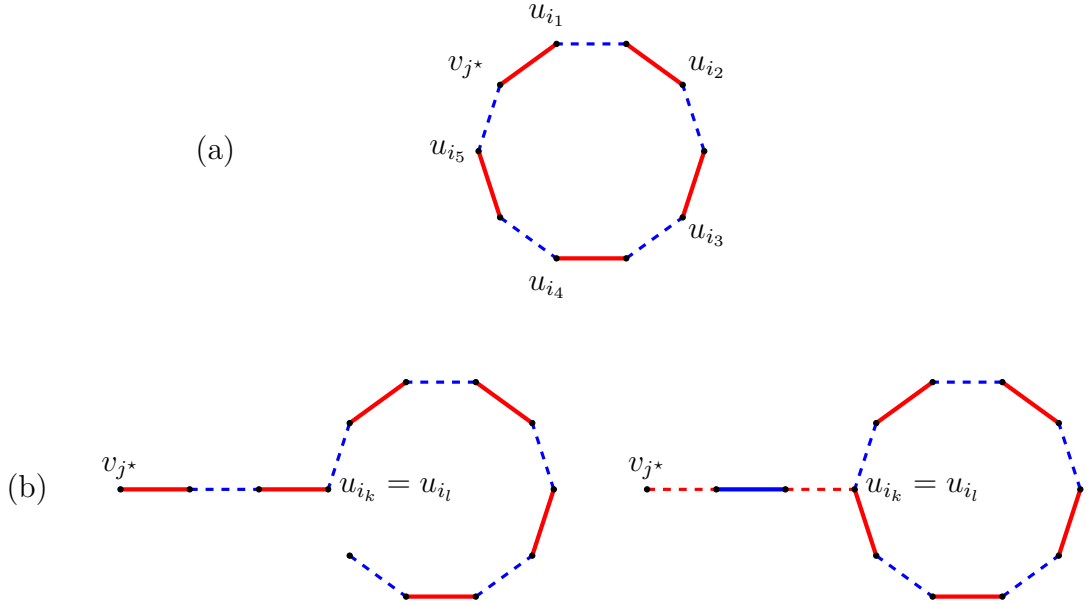

We finally relate the execution trace of the Lemke--Howson algorithm when it is applied to an instance of~\eqref{NECP} that satisfies the dominance condition, and to its logarithmic image. Proposition~\ref{prop:classical_basis} ensures that if $w = \cla M z + q, \, (w, z) \in \R_{\geq 0}^{n + n}$ satisfies the dominance condition, then this system as well as its logarithmic image are  nondegenerate. Therefore, the Lemke--Howson algorithm can be applied to both instances.
\begin{theorem}\label{th:classical_LH}
Consider an instance of~\eqref{NECP} that satisfies the dominance condition, and the instance of~\eqref{TNECP} given by its logarithmic image. The sequence of bases over which the Lemke--Howson algorithm iterates does not depend on whether it is applied to the classical instance or the tropical one. In particular, it returns a basis of a solution of the classical instance within at most $2n-1$ iterations. 
\end{theorem}

\begin{proof}
The Lemke--Howson algorithm starts from the same basis in both instances. Suppose that the two executions are at basis $B$, and that the same entering column $\gamma$ is selected. Let $B'$ be the next classical feasible basis visited. Then $B$ and $B'$ are adjacent and both are nondegenerate tropical feasible bases by Proposition~\ref{prop:classical_basis}. 
The basis $B'$ is the next tropical feasible basis visited by the Lemke--Howson algorithm thanks to Lemma~\ref{lemma:trop_pivot}, because the latter is a tropical basis (Proposition~\ref{prop:classical_basis}) included in $B \cup \{\gamma\}$. This implies that the leaving column is the same in both executions. Besides, the next entering column (if any) is defined as the twin of the latter in both executions.

Therefore, by Theorem~\ref{th:complexity}, the Lemke--Howson algorithm terminates in at most $2n-1$ iterations with a fully labeled basis that contains a red-colored element. 
\end{proof}

\section{Tropical Nash equilibria}\label{sec:tropical_nash}

Our original motivation to study the tropical Nash equilibrium complementarity problem was that the  corresponding classical problem contains the computation of Nash equilibria for bimatrix games. We provide more details on this relation, and discuss to which extent this can give rise to notions of tropical Nash equilibria.

Given two matrices $\nashA, \nashB \in \R^{r \times s}$, a \emph{bimatrix game} is a game in which two players play simultaneously by choosing actions $i \in [r]$ and $j \in [s]$ respectively, in which case their payoffs are respectively $\nashA_{ij}$ and $\nashB_{ij}$. \emph{Mixed strategies} of the two players are defined as probability measures over the set of their actions. They are represented by vectors $x \in \Delta_{r-1}$ and $y \in \Delta_{s-1}$, where $\Delta_p \coloneqq  \{ z \in \R_{\geq 0}^{p+1} \colon \sum_{k = 1}^{p+1} z_k = 1 \}$ is the standard $p$-dimensional simplex. A \emph{Nash equilibrium} of the bimatrix game is a pair of mixed strategies $(x^*, y^*) \in \Delta_{r-1} \times \Delta_{s-1}$ verifying 
\begin{equation}\label{eq:nash}
\begin{aligned}
\trans{x^*} \cla \nashA y^* & \geq \trans{x} \nashA y^* \, , \\
\trans{x^*} \cla \nashB y^* & \geq \trans{x^*} \nashB y \, ,
\end{aligned}
\end{equation}
for all $(x,y) \in \Delta_{r-1} \times \Delta_{s-1}$. This means that the mixed strategy $x^*$ of Player 1 is a best response to the strategy $y^*$ of Player 2, and vice versa. An equivalent alternative definition of Nash equilibria is the following: a Nash equilibrium is a pair $(x^*, y^*) \in \Delta_{r-1} \times \Delta_{s-1}$ verifying:
\begin{equation}\label{eq:pre_compl}
\begin{aligned}
\forall i \in [r] \, , && x^*_i > 0 &\implies \forall k \in [r] \, , \; (\nashA y^*)_i \geq (\nashA y^*)_k \, , \\
\forall j \in [s] \, , && y^*_j > 0 &\implies \forall l \in [s] \, , \; (\trans{\nashB} x^*)_j \geq (\trans{\nashB} x^*)_l \, .
\end{aligned}
\end{equation}
Observe that, under any of these two definitions, the set of Nash equilibria remains the same when we shift all the entries of $\nashA$ and $\nashB$ by a constant. Therefore, we can assume that $\nashA$ and $\nashB$ are nonnegative (entrywise), and every column of $\nashA$ and $\trans{\nashB}$ has at least one positive entry. Then, using homogenization and introducing slack variables, it can be shown that Nash equilibria are in one-to-one correspondence with the solutions of~\eqref{NECP} where $\cla M = -\begin{psmallmatrix}
0 & \nashA \\
\trans{\nashB} & 0 	
\end{psmallmatrix}$ and $q$ is the all-$1$ vector of size $r+s$~\cite{stengel_2007}. We are now ready to prove \Cref{prop:spec-poly}.

\begin{proof}[\proofname{} of Proposition~\ref{prop:spec-poly}]
Let $A \coloneqq \begin{pmatrix} I & -M \end{pmatrix}$. By Theorems~\ref{thm:tnecp} and~\ref{th:support}, it suffices to check that the system $A \begin{psmallmatrix} w \\ z \end{psmallmatrix} = q$ satisfies the dominance condition. 

The matrix $\norm{A}$ writes as $\begin{pmatrix} I & \begin{smallmatrix} 0 & \nashA' \\ \trans{\nashB'} & 0 \end{smallmatrix} \end{pmatrix}$, where the nonzero entries of $\nashA'$ (resp.~$\nashB'$) distinct from $1$ are less than $1/(r-1)$ (resp.~$1/(s-1)$). Item~\ref{item:a} is satisfied by considering the identity block in $\norm A$. For Item~\ref{item:b}, consider a $(r+s) \times (r+s)$-submatrix $M'$ dominating a permutation matrix. As each column of $\norm A$ has exactly one $1$-entry, the matrix $M'$ must have exactly one $1$-entry on each row: more than two $1$-entries in a row of $M'$ would leave another row without any $1$-entry (which is impossible as $M'$ dominates a permutation matrix). Moreover, the matrix $M'$ has at most $r$ columns taken from the block $\begin{psmallmatrix} \nashA' \\ 0 \end{psmallmatrix}$, and even at most $r-1$ such columns if $M'$ contains a column from the first $r$ columns of the identity block of $\norm A$. Thus, for all $i \in [r]$, the $i$th row of $M'$ has exactly one $1$-entry, at most $r-1$ entries of $P'$ less than $1$, and $0$-entries. Therefore, the sum of such row is less than $2$. The argument is analogous for rows of $M'$ indexed by some $i$ where $r+1 \leq i \leq r+s$. 
\end{proof}

We now define tropical Nash equilibria. Let $\Delta^\trop_p \coloneqq \{ z \in \T^{p+1} \colon \tsum_{k = 1}^{p+1} z_k = 0 \}$ be the tropical standard simplex. Given $\nashA, \nashB \in \T^{r \times s}$, a \emph{tropical Nash equilibrium} is a pair $(x^*,y^*) \in \Delta^\trop_{r-1} \times \Delta^\trop_{s-1}$ such that
\begin{equation}\label{eq:trop_pre_compl}
\begin{aligned}
\forall i \in [r] \, , && x^*_i > -\infty &\implies \forall k \in [r] \, , \;(\nashA \tdot y^*)_i \geq (\nashA \tdot y^*)_k \, , \\
\forall j \in [s] \, , && y^*_j > -\infty & \implies \forall l \in [s] \, , \; (\trans{\nashB} \tdot x^*)_j \geq (\trans{\nashB} \tdot x^*)_l \, .
\end{aligned}
\end{equation}
The next proposition shows that the correspondence between the Nash equilibria in the sense of~\eqref{eq:pre_compl} and the solutions of~\eqref{NECP} remains valid in the tropical setting. The proof is analogous to the classical case sketched above. The correspondence is built via a normalization we define now. 
For $x \in \T^r$ and $y \in \T^s$ having each at least one entry distinct from $-\infty$, the \emph{normalization} of $z=(x,y)$ is the pair of scaled elements $(x^*,y^*) \in \Delta^\trop_{r-1} \times \Delta^\trop_{s-1}$ where $x^* = \alpha \tdot x$ and $y^* = \beta \tdot y$ for some $\alpha, \beta \in \T$.

\begin{proposition}\label{prop:correspondence}
Suppose that every column of $\nashA$ and $\trans{\nashB}$ has at least one entry distinct from $-\infty$, and consider~\eqref{TNECP} with $M^- = \begin{psmallmatrix}
-\infty & \nashA \\
\trans{\nashB} & -\infty
\end{psmallmatrix}$ and $q^+$ the all-$0$ vector of size $r+s$. 
Then, the normalization maps bijectively the projection $z$ of the solutions of~\eqref{TNECP} to the pairs $(x^*, y^*) \in 	\Delta^\trop_{r-1} \times \Delta^\trop_{s-1}$ that satisfy~\eqref{eq:trop_pre_compl}.
\end{proposition} 

\begin{proof}
Given $\gamma \neq -\infty$, we denote $\gamma^{\tdot (-1)} \coloneqq -\gamma$. 
We introduce the sets:
\begin{align*}
\Qcal & \coloneqq \Big\{(x, u, \lambda) \in \T^r \times \T^s \times \T \colon \tsum_i x_i = 0 \; \text{and} \; u \tplus \trans{\nashB} \tdot x = \begin{psmallmatrix} \lambda \\[-1ex] \vdots \\ \lambda \end{psmallmatrix} \Big\} \, ,\\
\Pcal & \coloneqq \Big\{(y, v, \mu) \in \T^r \times \T^s \times \T \colon \tsum_j y_j = 0 \; \text{and} \; v \tplus \nashA \tdot y = \begin{psmallmatrix} \mu \\[-1ex] \vdots \\ \mu \end{psmallmatrix}\Big\} \, .
\end{align*}
We can check that $(x^*, y^*) \in \Delta^\trop_{r-1} \times \Delta^\trop_{s-1}$ satisfies~\eqref{eq:trop_pre_compl} if and only if there exist $(u, \lambda)$ and $(v, \mu)$ such that
\begin{equation}\label{eq:equiv_criterion}
(x^*,u,\lambda) \in \Qcal \qquad (y^*,v,\mu) \in \Pcal \qquad \trans{u} \tdot y^* = \trans{v} \tdot x^* = -\infty \, .
\end{equation}
Moreover, if~\eqref{eq:equiv_criterion} holds, then $\lambda = \max_l (\trans{\nashB} \tdot x^*)_l$. Indeed, the inequality $\lambda \geq \max_l (\trans{\nashB} \tdot x^*)_l$ is obvious, and the other inequality comes from the fact that $y^*_j > -\infty$ for some $j$, so that $u_j = -\infty$ and $\lambda = (\trans{\nashB} \tdot x^*)_j$. In a similar way, we can prove that \eqref{eq:equiv_criterion} implies $\mu = \max_k (\nashA \tdot y^*)_k$. Since no column of $\nashA$ and $\trans{\nashB}$ are equal to $-\infty$, we deduce that $\lambda > -\infty$ and $\mu > -\infty$ as soon as~\eqref{eq:equiv_criterion} is satisfied. 

We now consider a solution $w = (\bar{v}, \bar{u})$ and $z = (x,y)$ of~\eqref{TNECP}, and introduce $\alpha \coloneqq (\tsum_i x_i)^{\tdot (-1)}$ and $\beta \coloneqq (\tsum_j y_j)^{\tdot (-1)}$. This is well-defined, since if $x$ were the $-\infty$ vector, then we would have $\bar{u} = q^+$, which would imply that $y$ and thus $z$ are $-\infty$ vectors. Similarly, we can show $y \neq -\infty$. Then, it is immediate to check that $(x^*, u, \lambda) \coloneqq \alpha \tdot (x, \bar u, 0)$ and $(y^*, v, \mu) \coloneqq \beta \tdot (y, \bar v, 0)$ satisfy \eqref{eq:equiv_criterion}. We deduce that the normalization is well-defined, and maps the projection $z$ of the solutions of~\eqref{TNECP} to the pairs $(x^*, y^*) \in \Delta^\trop_{r-1} \times \Delta^\trop_{s-1}$ satisfying~\eqref{eq:trop_pre_compl}. Moreover, it is injective. Indeed, if $(x,y)$ is sent to $(x^*, y^*)$, then $\alpha = (\tsum_i x_i)^{\tdot (-1)}$ and $\beta = (\tsum_i y_j)^{\tdot (-1)}$ respectively satisfy $\alpha = \max_l (\trans{\nashB} \tdot x^*)_l$ and $\beta = \max_k (\nashA \tdot y^*)_k$ by the previous discussion. 

It remains to prove that the map is surjective. Take $(x^*, y^*)$ satisfying~\eqref{eq:equiv_criterion} for some $u, v, \lambda, \mu$. Then, $w = (\mu^{\tdot (-1)} \tdot v, \lambda^{\tdot (-1)} \tdot u)$ and $z = (\lambda^{\tdot (-1)} \tdot x^*, \mu^{\tdot (-1)} \tdot y^*)$ form a solution of~\eqref{TNECP}, and $z$ is sent by the normalization map to $(x^*, y^*)$.
\end{proof}

The Quint--Shubik conjecture in game theory~\cite{quint1997theorem}, disproved by von Stengel~\cite{von1999new}, claimed that every nondegenerate bimatrix game has at most $2^r-1$ equilibria, when the two players have the same number $r$ of actions. The following proposition states that the conjecture in true in the tropical setting:
\begin{proposition}\label{prop:quint-shubik}
let $P, Q \in \T^{r \times r}$ such that every column of $P$ and $\trans Q$ has at least one entry distinct from $-\infty$. Suppose that every column of $P$ and $\trans Q$ has a unique maximizing entry. Then, there are at most $2^r-1$ tropical Nash equilibria.
\end{proposition}

\begin{proof}
Since every column of $P$ and $\trans Q$ has a unique maximizing entry, the corresponding instance of~\eqref{TNECP} (in the sense of \Cref{prop:correspondence}) is nondegenerate. Let $G$ be the associated graph (as introduced in \Cref{sec:tnecp}). It consists of the row nodes $u_i$ and column nodes $v_i$ where $i \in [2r]$. Let $i \in [r]$. The row node~$u_i$ is connected to the column node $v_i$ by a blue edge. Recall that the latter is incident to at least one red edge. Due to the antidiagonal shape of the matrix $M^- = \begin{psmallmatrix}
-\infty & \nashA \\
\trans{\nashB} & -\infty
\end{psmallmatrix}$, any such red edge connects $v_i$ to some row node $u_j$ where $r+1 \leq j \leq 2r$. Finally, the latter node is connected by a blue edge to the column node $v_j$. As a consequence, every connected component of $G$ contains at least four distinct nodes. This ensures that the number of its connected components is at most $r$. We deduce the result from \Cref{thm:count}.
\end{proof}

\begin{remark}
\Cref{th:support} together with \Cref{prop:quint-shubik} show in turn that classical bimatrix games where the two players have the same number of actions and where the matrix $M$ of the corresponding \eqref{NECP} formulation verifies the ``dominance condition'' satisfy the Quint--Shubik conjecture.
\end{remark}

Observe that our definition of tropical Nash equilibria fits the case where the payoff matrices are ``tropically nonnegative''  (every element of $\T$ is greater than or equal to the tropical zero). For bimatrix games where the payoffs are nonnegative elements of a real-closed field with a nonarchimedean valuation (see~\cite{AllamigeonGaubertSkomraDCG20} for a discussion), the valuation of Nash equilibria are tropical Nash equilibria. Moreover, provided that the valuations of the payoffs are generic, every tropical Nash equilibrium arises in this way.

It is remarkable that the definition of tropical Nash equilibria is not equivalent to the tropical analogue of~\eqref{eq:nash} (studied in a more general framework by Briec and Yesilce~\cite{Briec2022}), which we write as:
\begin{equation}\label{eq:tropical_nash}
\begin{aligned}
\trans{x^*} \tdot \nashA \tdot y^* & \geq \trans{x} \tdot \nashA \tdot y^* \, , \\
\trans{x^*} \tdot \nashB \tdot y^* & \geq \trans{x^*} \tdot \nashB \tdot y \, ,
\end{aligned}
\end{equation}
for all $(x,y) \in \Delta^\trop_{r-1} \times \Delta^\trop_{s-1}$. Indeed, $(x^*, y^*) = (0,0)$ is always a solution of~\eqref{eq:tropical_nash}. In fact, \eqref{eq:trop_pre_compl} implies~\eqref{eq:tropical_nash} but the converse $\eqref{eq:tropical_nash} \implies \eqref{eq:trop_pre_compl}$ is not true anymore. Indeed, the corresponding classical implication $\eqref{eq:nash} \implies \eqref{eq:pre_compl}$ relies on the fact that in a linear program, if a strict convex combination of points is an optimal solution, then all these points are optimal. This property does not hold in the tropical setting.

Extensions of the tropical semiring with ``negative'' elements, such as the symmetrized tropical semiring~\cite{PlusCDC90}, have been investigated; see also~\cite{AGGuterman2014Cramer}, and~\cite{LohoVeghITCS20} for a recent development of tropical convexity with signs. However, the two definitions of classical Nash equilibria (\ie, in the sense of~\eqref{eq:nash} and~\eqref{eq:pre_compl}) cannot be reduced to the case with nonnegative payoffs in the tropical setting (\eg, the shift of the payoffs is irreversible, and leads to non-equivalent problems in which every payoff is ultimately replaced by a large constant). As a consequence, the tropicalization of Nash equilibria with signed payoffs deserves further investigations. We expect that these have a different status complexitywise. Indeed, similarly to the classical case~\cite{Adler2013}, tropical signed matrix should be expressible enough to capture tropical feasibility problems with signed entries. The latter are equivalent to the so-called mean payoff games~\cite{AkianGaubertGuterman2012}, and the existence of polynomial time algorithm to solve them is an open question.

\bibliographystyle{amsplain}
\providecommand{\bysame}{\leavevmode\hbox to3em{\hrulefill}\thinspace}
\providecommand{\MR}{\relax\ifhmode\unskip\space\fi MR }
\providecommand{\MRhref}[2]{%
  \href{http://www.ams.org/mathscinet-getitem?mr=#1}{#2}
}
\providecommand{\href}[2]{#2}

\end{document}